\documentclass{siamltexmm}
\usepackage[utf8x]{inputenc}
\usepackage[OT1]{fontenc}

\usepackage{amsmath, latexsym, amsfonts, amssymb, amscd,bbm}
\usepackage[english]{babel}
\usepackage[dvips]{graphicx}

\renewcommand{\tilde}{\widetilde}

\def \R{\mathbb R}

\newtheorem{rem}{Remark}

\newcommand{\e}{{\mathrm e}}

\newcommand{\norm}[1]{\left \| #1 \right \|}
\newcommand{\x}{u}

\renewcommand{\tilde}{\widetilde}

\newcommand{\M}{\mathfrak{M}}

\newcommand{\RR}{\mathfrak{R}}
\newcommand{\G}{{\mathcal G}}

\newcommand{\slugmaster}{\slugger{siads}{}{}{}}

\numberwithin{equation}{section}

\title{A variational method for analyzing stochastic limit cycle oscillators\thanks{PCB and JNM were supported
by the National Science Foundation (DMS-1613048).}}
\author{Paul C. Bressloff\thanks{Department of Mathematics, University of Utah, Salt Lake City, UT 84112 USA  ({\tt bressloff@math.utah.edu}).} \and James N. MacLaurin\thanks{Department of Mathematics, University of Utah, Salt Lake City, UT 84112 USA  ({\tt maclaurin@math.utah.edu}).}}

\begin{document}

\maketitle

\begin{abstract}

We introduce a variational method for analyzing limit cycle oscillators in $\R^d$ driven by Gaussian noise. This allows us to derive exact stochastic differential equations (SDEs) for the amplitude and phase of the solution, which are accurate over times over order $\exp\big(Cb\epsilon^{-1}\big)$, where $\epsilon$ is the amplitude of the noise and $b$ the magnitude of decay of transverse fluctuations. Within the variational framework, different choices of the amplitude-phase decomposition correspond to different choices of the inner product space $\R^d$. For concreteness, we take a weighted Euclidean norm, so that the minimization scheme determines the phase by projecting the full solution on to the limit cycle using Floquet vectors. Since there is coupling between the amplitude and phase equations, even in the weak noise limit, there is a small but non-zero probability of a rare event in which the stochastic trajectory makes a large excursion away from a neighborhood of the limit cycle. We use the amplitude and phase equations to bound the probability of it doing this: finding that the typical time the system takes to leave a neighborhood of the oscillator scales as $\exp\big(Cb\epsilon^{-1}\big)$.

\end{abstract}

\begin{keywords}
stochastic oscillators
\end{keywords}
\begin{AMS}
60H20,60H25,92C20,92C15,92C17
\end{AMS}
\renewcommand{\thefootnote}{\fnsymbol{footnote}}

\section{Introduction}

A well-studied problem in dynamical systems theory is the construction and analysis of phase equations for stochastic limit cycle oscillators \cite{Erm10,nakao2016}. For example, consider the Ito stochastic differential equation (SDE) on $\R^d$,
\begin{equation}
du =  F(u)dt +  \sqrt{\epsilon}G(u) dW\label{Wdep}
\end{equation}
where $\epsilon >0$ determines the noise strength and $W_t$ is a vector of (correlated) Brownian motions with covariance $Q \in \mathbb{R}^{d \times d}$,
\[
\mathbb{E}\big[W(t) W^{\top}(t) \big] = tQ.
\]
Suppose that the deterministic equation for $\epsilon =0$,
\begin{equation}
\label{det}
\frac{du}{dt} =  F(u),\quad u \in \R^d
\end{equation}
with $F\in C^2$ has a stable periodic solution $u=U(t)$ with $U(t)=U(t+\Delta_0)$, where
$\omega_0=2\pi/\Delta_0$ is the natural frequency of the oscillator. In state space the solution is an isolated attractive trajectory
called a limit cycle.
The dynamics on the limit cycle can be described by a uniformly rotating phase such that
\begin{equation}
\frac{d\theta}{dt}=\omega_0,
\end{equation}
and $u={\Phi}(\theta(t))$ with ${\Phi}$ a $2\pi$-periodic function. Note that the phase is 
neutrally stable with respect to perturbations along the limit cycle -- this reflects invariance of an autonomous
dynamical system with respect to time shifts. Turning to the SDE (\ref{Wdep}), let us assume that
the noise amplitude $\epsilon$ is sufficiently small given the rate of attraction to the limit cycle, so that deviations transverse to the
limit cycle are also small (up to some exponentially large stopping time). This suggests that the definition of a phase variable persists in the stochastic setting, and one can derive a stochastic phase equation. However, there is not a unique way to define the phase, which has led to two complementary methods for obtaining a stochastic phase equation: (i) the method of isochrons \cite{Teramae04,Goldobin05,Nakao07,Yoshimura08,Teramae09}, and (ii) an explicit amplitude-phase decomposition \cite{Gonze02,Koeppl11,Bonnin17}. (See also the recent survey by Ashwin et al \cite{Ashwin16}.) A major point to note is that while many of the current definitions of the stochastic phase are only accurate on timescales of $O(\epsilon^{-1})$, the oscillator will typically stay in a neighborhood of the limit cycle for times of order $O\big(\exp(Cb\epsilon^{-1})\big)$ (where $b$ is the rate of decay of transverse fluctuations), and it is therefore desirable to have a definition of the phase on these much longer timescales. This is particularly important, since many of the cited papers are explicitly trying to study the long-time ergodic behavior of the oscillator.

In this paper, we introduce a variational method for carrying out the amplitude-phase decomposition, which yields exact SDEs for the amplitude and phase, similar to those recently obtained in \cite{Bonnin17} using the implicit function theorem. In addition to simplifying the derivation of these equations, the variational method provides a more general framework for analyzing stochastic dynamical systems with marginally stable degrees of freedom, see for example \cite{inglis16,lang2017finite,gottwald2017finite}. Within the variational framework, different choices of phase correspond to different choices of the inner product space $\R^d$. For concreteness, we take a weighted Euclidean norm, so that the minimization scheme determines the phase by projecting the full solution on to the limit cycle using Floquet vectors. Hence, in a neighborhood of the limit cycle the phase variable coincides with the isochronal phase \cite{Bonnin17}. This has the advantage that the amplitude and phase decouple to linear order. 

In addition, our variational method provides an explicit analytic expression for the phase SDE, which is accurate for exponentially long times, as well as a precise formula  for determining the phase given any particular realization of the SDE for $u_t$. More significantly, since the stochastic phase and amplitude do couple even in the weak noise limit, there is a small but non-zero probability of a rare event in which the stochastic trajectory makes a large excursion away from an $O(\epsilon^\rho)$ neighborhood of the limit cycle, for any $\rho < 1/2$. In this paper, we use the exact amplitude and phase equations to derive strong exponential bounds on the growth of transverse fluctuations. More precisely, we show that the expectation of the time it takes to leave an $O(\epsilon^{\rho})$ neighbourhood of the limit cycle scales as $\exp\big(Cb\epsilon^{2\rho-1}\big)$, for a constant $C$, where $b$ is the magnitude of decay of the transverse fluctuations. These bounds are thus very useful in both the small noise limit, and the limit of strong decay of transverse fluctuations (as discussed in \cite{Teramae09,Newby14}). Indeed they are accurate for finite $\epsilon / b$ and are more flexible and powerful than classical large deviations bounds. Our method is novel and uses a rescaling of time to demonstrate that the leading order behavior of the amplitude term is that of a stable Ornstein-Uhlenbeck Process. These bounds also mean that the SDE for the phase defined in \S 2 is well-defined for times of order $\exp\big(Cb / \epsilon\big)$.

 In the remainder of this section we briefly review the two main phase reduction methods. The variational formulation is introduced in \S 2, where we derive the exact amplitude and phase equations using Ito's lemma. In \S 3 we carry out a perturbation expansion in the weak noise limit and compare the resulting phase equation with previous versions. Finally, exponential bounds on transverse fluctuations are derived in \S 4.  
 
 \subsection{Isochrons and phase--resetting curves}

\begin{figure}[b!]
\begin{center}
\includegraphics[width=7cm]{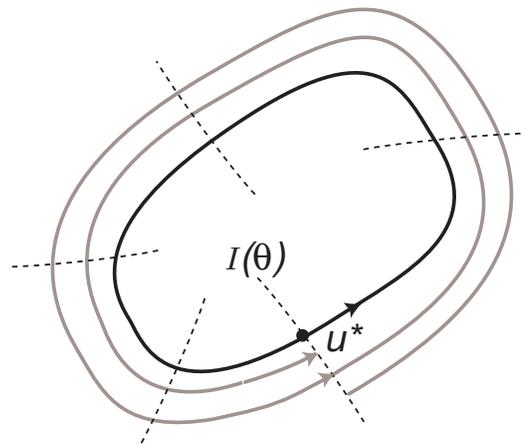}
\caption{\small Isochrones in the neighborhood of a stable limit cycle}
\label{isochrone}
\end{center}
\end{figure}

Suppose that we observe the unperturbed system (\ref{det}) stroboscopically at time intervals of length $\Delta_0$. This leads to a
Poincare mapping
\begin{equation*}
u(t)\rightarrow u(t+\Delta_0)\equiv {\mathcal P}(u(t)).
\end{equation*}
This mapping has all points on the limit cycle as fixed points. Choose a point $u^*$ on the cycle and consider
all points in the vicinity of $u^*$ that are attracted to it under the action of ${\mathcal P}$. They form a
$(d-1)$-dimensional hypersurface ${\mathcal I}$, called an isochron, crossing the limit cycle at $u^*$ (see Fig. \ref{isochrone}) \cite{Winfree80,Kuramoto84,Glass88,Holmes04}. A unique isochron can be drawn through each point on the limit cycle (at least locally) so the isochrons can be parameterized by the phase,
${\mathcal I}={\mathcal I}(\theta)$. Finally, the definition of phase is extended by taking all points $\x\in {\mathcal I}(\theta)$ to have the same phase,
$\Theta(\x)=\theta$, which then rotates at the natural frequency $\omega_0$ (in the
unperturbed case). Hence, for an unperturbed oscillator in the vicinity of the limit cycle we have
\begin{eqnarray}
\omega_0 = \frac{d\Theta}{dt}=\sum_{k=1}^d\frac{\partial \Theta}{\partial u_k}\frac{du_k}{dt} =\sum_{k=1}^d\frac{\partial \Theta}{\partial u_k}F_k(u) .\nonumber
\end{eqnarray}
Now consider the deterministically perturbed system
\begin{equation}
\label{utt}
\frac{du}{dt} =  F(u)+\sqrt{\epsilon}G(u,t),
\end{equation}
where $G$ is a $\Delta$-periodic function of $t$, say.
Keeping the definition of isochrons for the unperturbed system, one finds that to leading order
\begin{equation}
\frac{d\Theta}{dt}=\sum_{k=1}^d\frac{\partial \Theta}{\partial u_k}(F_k(u)+\sqrt{\epsilon} G_k({u},t))=\omega_0+\sqrt{\epsilon}\sum_{k=1}^d\ \frac{\partial \Theta}{\partial u_k}G_k(u,t). \nonumber
\end{equation}
As a further leading order approximation, deviations of $u$ from the limit cycle are ignored. Hence, setting $u(t)=\Phi(\omega_0 t)$ with $\Phi$ the $2\pi$-periodic solution on the limit cycle,
\begin{equation}
\frac{d\Theta}{dt}=\omega_0+\sqrt{\epsilon}\sum_{k=1}^d \left .\frac{\partial \Theta}{\partial u_k}\right |_{u=\Phi}G_k(\Phi,t) .\nonumber
\end{equation}
Finally, since points on the limit cycle are in 1:1 correspondence with the phase $\theta$, one can set $U=U(\theta)$ and $\Theta(U(\theta))=\theta$ to obtain the closed
phase equation
\begin{equation}
\frac{d\theta}{dt}=\omega_0+\sqrt{\epsilon} \sum_{k=1}^d R_k(\theta)G_k(\Phi(\theta),t)
\label{cbphase}
\end{equation}
where
\begin{equation}
R_k(\theta)= \left . \frac{\partial \Theta}{\partial u_k}\right |_{u=\Phi(\theta)}
\label{Q2}
\end{equation}
is a $2\pi$-periodic function of $\theta$ known as the $k$th component of the phase response curve (PRC).

It is well known that the PRC $R(\theta)$ can also be obtained as a $2\pi$-periodic solution of the linear equation \cite{Erm96,Erm10,nakao2016}
\begin{equation}
\label{adj}
\omega_0\frac{dR(\theta)}{d\theta}=-J(\theta)^{\top}\cdot  R(\theta),
\end{equation}
with the normalization condition
\begin{equation}
R(\theta)\cdot \frac{d\Phi(\theta)}{d\theta}=1.
\end{equation}
Here $J(\theta)^{\top}$ is the transpose of the Jacobian matrix $J(\theta)$, i.e.
\begin{equation}
\label{Jac}
J_{jk}(\theta)\equiv \left . \frac{\partial F_j}{\partial u_k}\right |_{u=\Phi(\theta)}.
\end{equation}
 It should be noted that we can evaluate the multiplication of the Jacobian by the derivative of $\Phi$ by differentiating the unperturbed ODE on the limit cycle, 
\[\omega_0 \frac{d\Phi}{d\theta} = F(\Phi(\theta)),\]
with respect to $\theta$. This gives
\begin{equation}
\frac{d}{d\theta} \left (\frac{d\Phi}{d\theta}\right )=\omega_0^{-1}J( \theta)\cdot \frac{d\Phi}{d\theta},
\label{nadj}
\end{equation}
%with

The  next step is to assume that the above phase reduction procedure can also be applied to the SDE (\ref{Wdep}). This would then lead to the stochastic phase equation
\begin{equation}
d\theta=\omega_0dt +\sqrt{\epsilon} \sum_{k,l=1}^d R_k(\theta)G_{kl}(\Phi(\theta))dW_l(t).
\label{cbphase2}
\end{equation}
However, this does not take proper account of stochastic calculus as expressed by Ito's lemma \cite{Gardiner09}. That is, the phase reduction procedure assumes that the ordinary rules of calculus apply. In the stochstic setting, this only holds if the multiplicative white noise term in equations (\ref{Wdep}) and (\ref{cbphase2}) are interepreted in the sense of Stratonovich. However, the Ito form of the stochastic phase equation is more useful when calculating correlations, for example. Hence, converting equation (\ref{cbphase2}) from Stratonovich to Ito using Ito's lemma gives \cite{Yoshimura08,Teramae09}
\begin{equation}
d\theta=\left [ \omega_0+\epsilon \sum_{k=1}^d Z'_k(\theta)Q_{kl}Z_l(\theta)\right ]dt +\sqrt{\epsilon} \sum_{k,l=1}^d Z_k(\theta)dW_l(t),
\label{cbphase3}
\end{equation}
where we have set
\begin{equation}
\label{Z}
Z_l(\theta)=\sum_{k=1}^dR_k(\theta)G_{kl}(\Phi(\theta)).
\end{equation}
Hence, Ito's lemma yields an $O(\epsilon)$ contribution to the phase drift.  Another subtle feature of the stochastic phase reduction procedure is that another $O(\epsilon)$ contribution occurs when taking into account perturbations transverse to the limit cycle \cite{Yoshimura08}. However, the latter contribution is negligible if the limit cycle is strongly attracting \cite{Teramae09}. 

\subsection{Amplitude-phase decomposition}

\begin{figure}[b!]
\centering
\includegraphics[width=8cm]{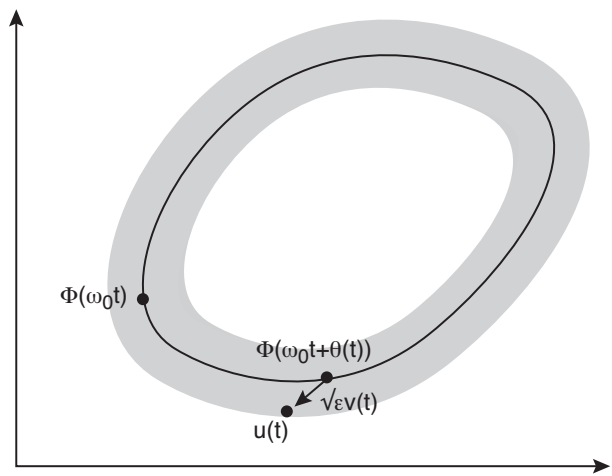}
\caption{Decomposition of the stochastic solution $u(t)$ into a random phase shift $\theta(t)$ along the deterministic limit cycle and a random transversal component $\sqrt{\epsilon}v(t)$.}
\label{scycle}
\end{figure}

\begin{figure}[t!]
\centering
\includegraphics[width=15cm]{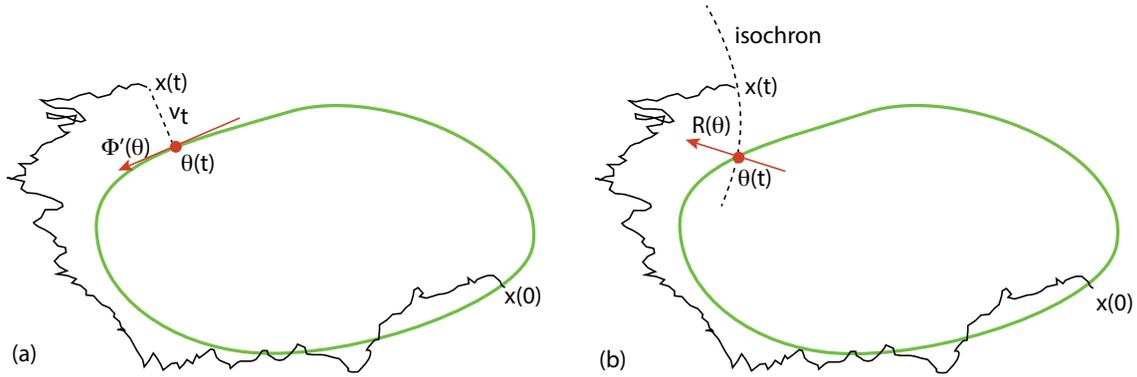}
\caption{The two different projection schemes highlighted in Ref. \cite{Bonnin17}. (a) Orthogonal projection with respect to the Euclidean norm of the solution $x(t)$ at time $t$ on to the limit cycle. Response to perturbations depends on the tangent vector to the limit cycle, $\Phi'(\theta)$ (b) The method of isochrons determines the phase $\theta(t)$ by tracing where the isochron through $x(t)$ intersects the limit cycle.The response to perturbations depends on the phase response curve ${R}(\theta)$, which is normal to the isochron at the point of intersection with the limit cycle.}
\label{fig:var}
\end{figure}

An alternative way to derive a stochastic-phase equation is to explicitly decompose the solution of (\ref{Wdep}) into longitudinal (phase) and transverse (amplitude) fluctuations of the limit cycle \cite{Boland09,Koeppl11,Bonnin17}. The basic intuition is that Gaussian-like transverse fluctuations are distributed in a tube of radius $1/\sqrt{\epsilon}$ (up to some stopping time), whereas the phase around the limit cycle undergoes Brownian diffusion. Thus, the solution is decomposed in the form
\begin{equation}
\label{defcon}
u(t)=\Phi(\omega_0t+\theta(t))+\sqrt{\epsilon}v(t),
\end{equation}
where the scalar random variable $\theta(t)$ represents the undamped random phase shift along the limit cycle, and $v(t)$ is a transversal perturbation, see Fig. \ref{scycle}. Since there is no damping of fluctuations along the limit cycle, the random phase $\theta(t)$ is taken to undergo Brownian motion. However, it is important to note that the decomposition (\ref{defcon}) is not unique, so that the precise definition of the phase depends on the particular method of analysis. 
For example, one study defines the phase so that there is no drift \cite{Koeppl11}. On the other hand, Gonze et al. \cite{Gonze02} focus on determining an effective phase diffusion coefficient based on a WKB approximation of solutions to the corresponding Fokker-Planck equation. Finally, Bonnin \cite{Bonnin17} combines an amplitude-phase decomposition with Floquet theory to show that if Floquet vectors are used, then the resulting phase variable in a neighborhood of the limit cycle coincides with the asymptotic phase based on isochrons, see Fig. \ref{fig:var}.

\section{Variational method}
Suppose that the deterministic ODE
\begin{equation}
\frac{du_t}{dt} =  F(u_t),\quad u_t \in \R^d
\end{equation}
supports a stable periodic solution of the form $u_t=\Phi(\omega_0 t)$ with $\Phi(\omega_0 t+2\pi n) =\Phi(t)$ for all integers $n$, and $\Delta_0 =2\pi/\omega_0$ is the fundamental period of the oscillator. We are interested in deriving a stochastic equation for the effective phase of the oscillator when the system is perturbed by weak noise. Therefore,
consider the Ito SDE\footnote{It would be straightforward to extend the results of the paper if we were to interpret the stochastic integrals in the Stratonovich sense.}
\begin{equation}
du_t =  F(u_t)dt +  \sqrt{\epsilon}G(u_t) dW_t\label{dep2}
\end{equation}
where $\epsilon >0$ determines the noise strength.
Here $W_t$ is a vector of (potentially correlated) Brownian motions with covariance $Q \in \mathbb{R}^{d \times d}$,
\[
\mathbb{E}\big[W_t W_t^{\top} \big] = tQ.
\]
In the above, $G$ is a Lipschitz map from $\R^d \to \R^{d\times d}$. Throughout this paper, for any matrix $A$, $\norm{A}$ denotes the spectral norm. We assume a uniform bound on the spectral norm of $G$, i.e. there exists a constant $\lambda_G$ such that
\begin{equation}
\sup_{u \in \R^d}\norm{G(u)} \leq \lambda_G.
\end{equation}

In the presence of noise we wish to decompose the solution $u_t$ into two components: the `closest' point of $\Phi(\beta_t)$ to $u_t$ for a $\R^d$-valued process $\beta_t$, and an `error' $v_t$ that represents the amplitude of transversal fluctuations:
\begin{equation}
\label{dep}
u_t=\Phi(\beta_t)+\sqrt{\epsilon}v_t,\quad v_t \in \R^d.
\end{equation}
However, as pointed out in \S 1.2, such a decomposition is not unique unless we impose an additional mathematical constraint. We will adapt a variational principle previously introduced by Inglis and Maclaurin \cite{inglis16} within the context of traveling waves in stochastic neural fields. First, we must introduce a little Floquet Theory.

\subsection{Floquet decomposition and weighted norm} For any $0 \leq t$, define $\Pi(t) \in \R^{d\times d}$ to be the following Fundamental matrix for the ODE
\begin{equation}\label{eq: ODE Jz}
\frac{d z}{dt} = J(t)z
\end{equation}
for $J(t)=J (\omega_0 t)$. That is, $\Pi(t):=  \big( z_1(t) | z_2(t) | \ldots |z_d(t) \big)$, where $z_i(t)$ satisfies \eqref{eq: ODE Jz}, $z_1(0) = \Phi'(0)$, and $\lbrace z_i(0) \rbrace_{i=1}^d$ is an orthogonal basis for $\R^d$. Floquet Theory states that there exists a diagonal matrix $\mathcal{S}=\mbox{diag}(\nu_1,\ldots,\nu_d)$ whose diagonal entries are the Floquet characteristic exponents, such that
\begin{equation}
\label{pip}
\Pi\big( t\big) = P\big(\omega_0 t \big)\exp\big(t\mathcal{S} \big)P^{-1}(0),
\end{equation}
with $P(\theta)$ a $2\pi$-periodic matrix whose first column is $\Phi'(\omega_0t)$, and $\nu_1 = 0$. That is, $P(\theta)^{-1}\Phi'(\theta) ={\bf e}$ with ${\bf e}_{j}=\delta_{1,j}$. In order to simplify the following notation, we will assume throughout this paper that the Floquet multipliers are real and hence $P(\theta)$ is a real matrix. One could readily generalize these results to the case that $\mathcal{S}$ is complex. The limit cycle is taken to be stable, meaning that for a constant $b > 0$, for all $2\leq i \leq d$,
\begin{equation}
\nu_i \leq - b.
\end{equation}
It follows from the fact that $F \in \mathcal{C}^2$ and $P \in \mathcal{C}^2$. Furthermore $P^{-1}(\theta)$ exists for all $\theta$, since $\Pi^{-1}(t)$ exists for all $t$.

The above Floquet decomposition motivates the following weighted inner product: For any $\theta \in \R$, denoting the standard Euclidean dot product on $\mathbb{R}^d$ by $\langle \cdot , \cdot\rangle$,
\[
\langle u,v \rangle_\theta = \big\langle P^{-1}(\theta)u,P^{-1}(\theta)v \big\rangle,
\]
and $\norm{u}_\theta = \sqrt{\langle u,u\rangle_{\theta}}$. This weighting is useful for two reasons: it leads to a leading order separation of the phase from the amplitude (see \S 3), and it facilitates the strong bounds of \S 4 because the weighted amplitude always decays, no matter what the phase. The former is a consequence of the fact that the matrix $P^{-1}(\theta)$ generates a coordination transformation in which the phase in a neighborhood of the limit cycle coincides with the asymptotic phase defined using isochrons (see also \cite{Bonnin17}). This is reflected by the following relationship between the tangent vector to the limit cycle, $\Phi'(\theta)$, and
 the PRC $R(\theta)$ of equation (\ref{Q2}):
 \begin{equation}
\label{Rtan}
\M(\theta)P^{\top}(\theta){R}(\theta)=P^{-1}(\theta)\Phi'(\theta),
\end{equation}
where
\begin{equation}
\M(\theta):= \norm{P^{-1}(\theta)\Phi'(\theta)}^2.
\end{equation}

We will proceed by defining $R(\theta)$ according to equation (\ref{Rtan}) and showing that it satisfies the adjoint equation (\ref{adj}). We will need the relation
\begin{equation}\label{eq: derivative P 0}
\omega_0 P'(\theta) = J(\theta)P(\theta)- P(\theta)\mathcal{S},
\end{equation}
which can be obtained by differentiating \eqref{pip}.
Differentiating both sides of equation (\ref{Rtan}) with respect to $\theta$, we have
\begin{align}
\label{Rtan2}
\M'P^{\top}R+\M P^TR'+\M (P^{\top})'R=P^{-1}\Phi''+(P^{-1})'\Phi',
\end{align}
with
\begin{align*}
\M'=2\bigg \langle P^{-1}\Phi''+(P^{-1})'\Phi',P^{-1}\Phi'\bigg \rangle.
\end{align*}
Equation (\ref{eq: derivative P 0}) implies that
\begin{align*}
\omega_0 (P^{\top}(\theta))' = P^{\top}(\theta)J^{\top}(\theta)- \mathcal{S}P^{\top}(\theta)
\end{align*}
and
\[\omega_0(P^{-1}(\theta))'=-P^{-1}(\theta)J(\theta)+\mathcal{S}P^{-1}(\theta).\]
We have used the fact that ${\mathcal S}$ is a diagonal matrix and $P^{-1}P'+(P^{-1})'P=0$ for any square matrix. Substituting these identities in equation (\ref{Rtan2}) yields
\begin{align*}
&\M' P^{\top}R+\M P^T(R'+\omega_0^{-1}J^{\top}R)-\omega_0^{-1}\M \mathcal{S}P^{\top}R\\
&\qquad =P^{-1}[\Phi''-\omega_0^{-1}J\Phi']+\omega_0^{-1}\mathcal{S}P^{-1}\Phi',
\end{align*}
and
\[\M'=\bigg \langle P^{-1}[\Phi''-\omega_0^{-1}J\Phi']+\omega_0^{-1}\mathcal{S}P^{-1}\Phi',P^{-1}\Phi'\bigg \rangle.\]
Now note that $\Phi'$ satisfies equation (\ref{nadj}) and $\mathcal{S}P^{-1}\Phi'=0$. The latter follows from the condition $P(\theta)^{-1}\Phi'(\theta) ={\bf e}$ and ${\mathcal S}{\bf e}=\nu_1=0$. It also holds that $\M'(\theta)=0$. (In fact, for the specific choice of $P(\theta)$, we have $\M(\theta)=\langle {\bf e},{\bf e}\rangle =1$.) Finally, from the definition of $(R(\theta)$, equation (\ref{Rtan}), we deduce that $\mathcal{S}P^{\top}(\theta)R(\theta)=0$ and hence
\[\M P^T(R'+\omega_0^{-1}J^{\top}R)=0.\]
Since $P^T(\theta)$ is non-singular for all $\theta$, $R$ satisfies equation (\ref{adj}) and can thus be identified as the PRC.

\subsection{Defining the stochastic phase using a variational principle}
We can now state the variational principle for the stochastic phase: $\beta_t$ is determined by requiring $\beta_t=a_t(\theta_t)$, where $a_t(\theta_t)$ for a prescribed time dependent weight $\theta_t$ is a local minimum of the following variational problem:
\begin{equation}\label{minim}
\underset{a\in {\mathcal N}(a(\theta_t))} \inf\| u_t-\Phi(a)\|_{\theta_t} =\| u_t-\Phi(a_t(\theta_t))\|_{\theta_t} ,
\end{equation}
with ${\mathcal N}\big(a_t(\theta_t)\big)$ denoting a sufficiently small neighborhood of $a_t\big(\theta_t\big)$. The minimization scheme is based on the orthogonal projection of the solution on to the limit cycle with respect to the weighted Euclidean norm at some $\theta_t$. 
We will derive an exact SDE for $\beta_t$ (up to some stopping time) by considering the first derivative
\begin{equation}\label{eq: Gi definition}
\G_0(z,a,\theta):=\frac{\partial }{\partial a} \| z-\Phi(a)\|^2_{\theta}  =-2\left \langle z-\Phi(a),\Phi'(a)\right \rangle_{\theta}.
\end{equation}
At the minimum, 
\begin{equation}\label{eq: minimum G 0}
\G_0(u_t,\beta_t,\theta_t)=0.
\end{equation}
We stipulate that the location of the weight must coincide with the location of the minimum, i.e. $\beta_t = \theta_t$, so that $\beta_t$ must satisfy the implicit equation
\begin{equation}\label{eq: minimum G 2}
\G(u_t,\beta_t):=\G_0(u_t,\beta_t,\beta_t)=0.
\end{equation}
It will be seen that, up to a stopping time $\tau$, there exists a unique continuous solution to the above equation. Note that we could have defined $\beta_t$ according to
\begin{equation}\label{minim2}
\underset{a\in {\mathcal N}(\beta_t)} \inf\| u_t-\Phi(a)\|_{a} =\| u_t-\Phi(\beta_t)\|_{\beta_t} ,
\end{equation}
which might seem more intuitive. However to leading order in $(u_t-\Phi(\beta_t))$, the above two schemes are equivalent, and we prefer the former because it leads to simpler equations.

Define $\M(z,a) \in \R$ according to
\begin{align}
\M(z,a)&:=\frac{1}{2}\frac{\partial \G(z,a)}{\partial a}=\frac{1}{2}\left . \frac{\partial \G_0(z,a,\theta)}{\partial a}\right |_{\theta=a}+\frac{1}{2}\left . \frac{\partial \G_0(z,a,\theta)}{\partial \theta}\right |_{\theta=a\nonumber }\\
&= 1 - \left\langle z-\Phi(a), \Phi''(a)\right \rangle_{a}
 - \left\langle z-\Phi(a),\frac{d}{da}\big[P^{-\top}(a)P^{-1}(a)\big]\Phi'(a)\right \rangle,
\label{curve}
\end{align}
where we have used the fact that $\norm{\Phi'(a)}^2_{a}=1$, which we proved in the previous section. Assume that initially $\M(u_0,\beta_0)>0$. We then seek an SDE for $\beta_t$ that holds for all times less than the stopping time $\tau$
\begin{equation}\label{defn: tau}
\tau=\inf\{s\geq 0:  \M(u_s,\beta_s)=0\}.
\end{equation}
The implicit function theorem guarantees that a unique continuous $\beta_t$ exists until this time. It is a consequence of Theorem \ref{Theorem: Bound v small} in the next section that there exists constants $C,\tilde{C} > 0$ such that
\[
\mathbb{P}\bigg( \tau \leq \exp\big(C b \epsilon^{-1}\big)\bigg) \leq \exp\big(-\tilde{C}b\epsilon^{-1}\big),
\]
where we recall that $b$ is the lower bound on the rate of decay of the Floquet exponents.

In order to derive the SDE for $\beta_t$, we apply Ito's lemma to the identity 
\begin{equation}\label{dG ti zero}
d\G_{t} := d\G(u_t,\beta_t)=0,
\end{equation}
with $du_t$ given by equation (\ref{dep2}) and $d\beta_t$ taken to satisfy an SDE of the form
\begin{equation}
\label{dbet}
d\beta_{t}=V(u_t,\beta_t) dt+ \sqrt{\epsilon}\left \langle B(u_t,\beta_t ), G(u_t)dW_t \right \rangle_{\beta_t}  ,
\end{equation}   
for functions $V$ and $B$ that we determine below. Using the definition of $\G(u_t,\beta_t,\beta_t)$, $d\G_t$ is found to be
\begin{align}
d\G_{t}
=&-2\left \langle du_t,\Phi'(\beta_t)\right \rangle_{\beta_t} +\left .\frac{\partial \G_{t}}{\partial a}\right |_{a=\beta_t}d\beta_{t}+\left . \frac{1}{2}\  \frac{\partial^2 \G_{t}}{\partial a^2}\right |_{a=\beta_t}d\beta_{t}d\beta_{t}
-2 \left \langle du_t, \Phi''(\beta_t)d\beta_{t}\right \rangle_{\beta_t}\nonumber \\
&-2\left\langle du_t,\frac{d}{da}\big[P^{-\top}(a)P^{-1}(a)\big]\big|_{a=\beta_t}\Phi'(\beta_t)\right \rangle d\beta_t. 
\label{G}
\end{align}
Note that we only include the $dt$ contributions from the quadratic differential terms involving the products $du_td\beta_{t}$ and $d\beta_{t}d\beta_{t}$, which are also known as cross-variations. 
In particular, writing $K(u_t,\beta_t) = G^{\top}(u_t) [P(\beta_t)P^{\top}(\beta_t)]^{-1}$,
\begin{equation}
\label{dbet2}
d\beta_{t}d\beta_{t}\widehat{=}\epsilon  \left \langle K(u_t,\beta_t){B}(u_t,\beta_t),QK(u_t,\beta_t){B}(u_t,\beta_t)\right \rangle dt,\end{equation}
\begin{align}
\label{du}
 \left \langle du_t,\Phi''(\beta_t)d\beta_{t}\right \rangle_{\beta_t}&\widehat{=} \sqrt{\epsilon} \left \langle G(u_t)dW_t,\Phi''(\beta_t)d\beta_{t}\right \rangle_{\beta_t} \nonumber \\
 &\widehat{=} \epsilon \left \langle G(u_t)dW_t  ,\Phi''(\beta_t) \langle B(u_t,\beta_t),G(u_t)dW_t \rangle_{\beta_t}  \right \rangle_{\beta_t}  \nonumber \\
   &\widehat{=} \epsilon \left \langle  K(u_t,\beta_t)\Phi''(\beta_t),QK(u_t,\beta_t) {B}(u_t,\beta_t)\right \rangle  dt .
\end{align}
and
\begin{multline}
\left\langle du_t,\frac{d}{da}\big[P^{-\top}(a)P^{-1}(a)\big]\big|_{a=\beta_t}\Phi'(\beta_t)\right \rangle d\beta_t \\= \epsilon\left\langle G^{\top}(u_t)\frac{d}{da}\big[P^{-\top}(a)P^{-1}(a)\big]\big|_{a=\beta_t}\Phi'(\beta_t),QK(u_t,\beta_t)B(u_t,\beta_t)\right\rangle.
\end{multline}

Substituting equations (\ref{dbet}), (\ref{dbet2}) and (\ref{du}) into equation 
(\ref{G}) yields an SDE of the form
\begin{equation}
d\G_{t}={\mathcal V}(u_t,\beta_t)dt+\sqrt{\epsilon}\langle {\mathcal B}(u_t,\beta_t),G(u_t)dW_t \rangle_{\beta_t} .
\end{equation}
In order that \eqref{dG ti zero} is satisfied, we require that both terms on the right-hand side of the above equation are zero, which will determine $V$ and $B$. First, we have
\begin{multline*}
0:=\frac{1}{2}\left \langle {\mathcal B}(u_t,\beta_t),G(u_t)dW_t\right \rangle_{\beta_t}  =\M(u_t,\beta_t)\left \langle B(u_t,\beta_t),G(u_t)dW_t \right \rangle_{\beta_t}  \\
- \left \langle G(u_t)dW_t,\Phi'(\beta_t)\right \rangle_{\beta_t}  .
\end{multline*}
Since for all times less than $\tau$, $ \M(u_t,\beta_t)  > 0$, it follows that $\M^{-1}$ exists, and hence
\begin{equation}
\label{Bi}
B(u_t,\beta_t)= \M (u_t,\beta_t)^{-1}\Phi'(\beta_t).
\end{equation}
 Second,
\begin{equation}
0 := {\mathcal V}(u_t,\beta_t)dt=\left [\left .  \frac{\partial \G_{t}}{\partial a}\right |_{a=\beta_t}V(u_t,\beta_t)-2\left (\left \langle F(u_t),\Phi'(\beta_t)\right \rangle_{\beta_t}   dt+\epsilon \kappa(u_t,\beta_t)\right )\right ]dt,
\label{V}
\end{equation}
with
\begin{multline}
\epsilon \kappa(u_t,\beta_t)dt:=-\frac{1}{4} \left . \frac{\partial^2 \G_{t}}{\partial a^2}\right |_{a=\beta_t}d\beta_{t}d\beta_{t} 
\quad +  \left \langle du_t,\Phi''(\beta_t)d\beta_{t}\right \rangle_{\beta_t}\\ + \left\langle du_t,\frac{d}{da}\big[P^{-\top}(a)P^{-1}(a)\big]\big|_{a=\beta_t}\Phi'(\beta_t)\right \rangle d\beta_t.
\label{VV}
\end{multline}
The cross-variations (\ref{dbet2}) and (\ref{du}) can now be evaluated using equation (\ref{Bi}):
\begin{align}
\label{dbet3}
d\beta_{t}d\beta_{t}&\widehat{=} \epsilon \M(u_t,\beta_t)^{-2} \left \langle K(u_t,\beta_t)\Phi'(\beta_t),QK(u_t,\beta_t) \Phi'(\beta_t)  \right \rangle  dt, \\
 \left \langle du_t,\Phi''(\beta_t)d\beta_{t}\right \rangle_{\beta_t}  &\widehat{=}\epsilon  \M(u_t,\beta_t)^{-1}\left \langle K(u_t,\beta_t)\Phi''(\beta_t),QK(u_t,\beta_t)\Phi'(\beta_t) \right \rangle  dt ,\label{du2}
 \end{align}
 and
 \begin{multline}
\left\langle du_t,\frac{d}{da}\big[P^{-\top}(a)P^{-1}(a)\big]\big|_{a=\beta_t}\Phi'(\beta_t)\right \rangle d\beta_t \\= \epsilon \mathfrak{M}(u_t,\beta_t)^{-1}\left\langle G^{\top}(u_t)\frac{d}{da}\big[P^{-\top}(a)P^{-1}(a)\big]\big|_{a=\beta_t}\Phi'(\beta_t),QK(u_t,\beta_t)\Phi'(\beta_t)\right\rangle dt.
\end{multline}
Equations (\ref{V})--(\ref{du2}) determine the drift term $V$ so that
\begin{equation}
d\beta_t = \mathfrak{M}(u_t,\beta_t)^{-1}\left [\left (\big\langle  F(u_t),\Phi'(\beta_t) \big\rangle_{\beta_t}+\epsilon \kappa(u_t,\beta_t)\right )dt + \sqrt{\epsilon}\bigg\langle G(u_t)dW_t ,\Phi'(\beta_t) \bigg\rangle_{\beta_t}\right ],
\label{bee}
\end{equation} 
where
\begin{multline}
\kappa(u_t,\beta_t) :=    \mathfrak{M}(u_t,\beta_t)^{-1}  \left \langle K(u_t,\beta_t)\Phi''(\beta_t),QK(u_t,\beta_t) \Phi'(\beta_t) \right \rangle  \\ 
+ \mathfrak{M}(u_t,\beta_t)^{-1}\left\langle G^{\top}(u_t)\frac{d}{da}\big[P^{-\top}(a)P^{-1}(a)\big]\big|_{a=\beta_t}\Phi'(\beta_t),QK(u_t,\beta_t)\Phi'(\beta_t)\right\rangle\\
+  \frac{\mathfrak{M}(u_t,\beta_t)^{-2}}{2}\left [\bigg\langle u_t - \Phi(\beta_t), \Phi'''(\beta_t) \bigg\rangle_{\beta_t}  -\bigg\langle \Phi'(\beta_t),\Phi''(\beta_t) \bigg\rangle_{\beta_t}\right. \\ 
 +\left\langle u_t-\Phi(\beta_t),\frac{d^2}{da^2}\big[P^{-\top}(a)P^{-1}(a)\big]\big|_{a=\beta_t}\Phi'(\beta_t)\right \rangle\\
+2\bigg\langle u_t-\Phi(\beta_t), \frac{d}{da}\big[P^{-\top}(a)P^{-1}(a)\big]\big|_{a=\beta_t}\Phi''(\beta_t) \bigg\rangle\\
\left .-\bigg\langle \Phi'(\beta_t), \frac{d}{da}\big[P^{-\top}(a)P^{-1}(a)\big]\big|_{a=\beta_t}\Phi'(\beta_t) \bigg\rangle 
\right ]  \left\langle K(u_t,\beta_t)\Phi'(\beta_t),QK(u_t,\beta_t) \Phi'(\beta_t) \right \rangle.
\end{multline}
Finally, recall that the amplitude term $v_t$ satisfies
\begin{equation}
\sqrt{\epsilon}v_t=u_t-\Phi_{\beta_t}.
\end{equation}
Hence, applying Ito's lemma
\begin{align}
\sqrt{\epsilon}dv_t&=du_t-\Phi'(\beta_t)d\beta_t-\frac{1}{2}\Phi''(\beta_t)d\beta_td\beta_t \nonumber \\
&=\left [F(u_t)- \mathfrak{M}(u_t,\beta_t)^{-1}\Phi'(\beta_t)\left (\big\langle  F(u_t),\Phi'(\beta_t) \big\rangle_{\beta_t}+\epsilon \kappa(u_t,\beta_t)\right )\right . \nonumber\\
&\qquad \left . -\frac{\epsilon}{2} \Phi''(\beta_t)\M(u_t,\beta_t)^{-2} \left \langle K(u_t,\beta_t)\Phi'(\beta_t),QK(u_t,\beta_t) \Phi'(\beta_t) \right \rangle\right ]dt \nonumber\\
&\qquad \qquad + \sqrt{\epsilon}\left [G(u_t)dW_t - \M(u_t,\beta_t)^{-1}\Phi'(\beta_t)\big\langle G(u_t)dW_t,\Phi'(\beta_t) \big\rangle_{\beta_t}\right ],
\label{vee}
\end{align}
where we have used equation (\ref{dep2}), and the differentials $d\Phi_t=F(\Phi_t)dt$ and
$d\Phi_{\beta_t} = \Phi' d\beta_t + \frac{1}{2}\Phi'' d\beta_t d\beta_t$.

\section{Weak noise limit}

In order to obtain a closed equation for $\beta_t$ we carry out a perturbation analysis in the weak noise limit, and compare the variational phase equation with various versions of the phase equations previously derived using isochronal phase reduction methods, see \S 1.1. We demonstrate that the linearization of our phase equation is accurate over timescales of order $\epsilon^{-1}$. This means that the timescale over which the linearization of our phase equation is accurate is of the same order as the isochronal phase equation. It should be noted that, as we explain in more detail in \S 5, our method possesses the additional virtue of having an analytic SDE that is accurate over timescales of order $O(\exp(Cb\epsilon^{-1}))$, where $b$ is the rate of decay of transverse fluctuations.

Suppose that $0 <\epsilon \ll 1$ and set $u_t=\Phi(\beta_t)$ on the right-hand side of equation (\ref{bee}). That is, we drop any $v_t$-dependent terms. Setting $\beta_t=\theta$, we obtain the explicit stochastic phase equation
\begin{equation}
\label{phase}
d\theta=[\omega_0 +\epsilon \widehat{\kappa}(\theta)]dt+\sqrt{\epsilon} \bigg \langle
G(\Phi(\theta))dW_t,{R}(\theta)\bigg\rangle,
\end{equation}
with ${R}(\theta)$ identified as the normal to the isochron crossing the limit cycle at $\theta$, see Fig. \ref{fig:var}(b) and equation (\ref{Rtan}):
\begin{equation}\label{eq: PRC again}
{R}(\theta)= [PP^{\top}(\theta)]^{-1}\Phi'(\theta),\end{equation}
since $\M(\theta)=1$.
We have used the identity
\begin{equation}
\big\langle\Phi'(\theta) \big),R(\theta)\big\rangle= 1,
\label{Fbee}
\end{equation}
and $F(\Phi(\theta))=\omega_0 \Phi'(\theta)$. Equation (\ref{phase}) has a similar form to the isochronal phase equation (\ref{cbphase3}). However, in contrast to the latter, there is no $O(\epsilon)$ contribution to the drift of the form $\bigg \langle {Z}'(\theta),Q {Z}(\theta)\bigg\rangle $ since we take the noise in SDE (\ref{dep2}) to be Ito rather than Stratonovich. Thus, the $O(\epsilon)$ drift term $\widehat{\kappa}(\theta)$ in equation (\ref{phase}) is the analog of the contributions from transverse fluctuations identified in \cite{Yoshimura08,Teramae09}.

As highlighted by Bonnin \cite{Bonnin17}, although neglecting the coupling between the phase and amplitude dynamics by setting $v_t=0$ yields a closed equation for the phase, it does lead to imprecision at short and intermediate times. (Errors at longer times due to large deviations from the limit cycle will be addressed in \S 4.) Here we show that taking into account the amplitude coupling only results in $O(\epsilon)$ contributions to the drift, not $O(\sqrt{\epsilon})$. Neglecting $v_t$-independent $O(\epsilon)$ drift terms, equation (\ref{bee}) becomes
\begin{equation}
 d\beta_t =\bigg\langle  F(u_t),\RR(u_t,\beta_t)\bigg\rangle_{\beta_t} dt+\sqrt{\epsilon}\bigg\langle G(u_t)dW_t ,\RR(u_t,\beta_t) \bigg\rangle_{\beta_t} ,\label{bee22}
\end{equation}
where
\begin{equation}
\RR(u_t,\beta_t)=\mathfrak{M}(u_t,\beta_t)^{-1}\Phi'(\beta_t) .
\end{equation}
Suppose that we rewrite $\RR$ as a function $\widehat{\RR}$ of $\beta_t$ and $v_t$ using equation (\ref{curve}): $\RR(u_t,\beta_t)=\widehat{\RR}(v_t,\beta_t)$ with
\begin{align*}
\widehat{\RR}(v_t,\beta_t)  &=\bigg(1 - \sqrt{\epsilon}\bigg\langle v_t, \Phi''(\beta_t)\bigg\rangle_{\beta_t}- \sqrt{\epsilon}\left \langle v_t, \frac{d}{da}\big[{P(a)P^{\top}(a)}^{-1}\big]\big|_{a=\beta_t}\Phi'(\beta_t)\right \rangle \bigg)^{-1}\Phi'(\beta_t) 
\end{align*}
Let us define
\begin{equation}\label{defn: H v theta}
H(v,\theta)=\left \langle F(\Phi(\theta)+\sqrt{\epsilon}v),\widehat{\RR}(v,\theta)\right \rangle_{\theta}.
\end{equation}
In the phase equation (\ref{phase}) we set $v=0$ and used $H(0,\theta)=\omega_0$. Suppose that we now include higher-order terms by Taylor expanding $H(v,\theta)$ with respect to $v$. In particular, consider the first derivative
\begin{align*}
&\frac{\partial H}{\partial v}(0,\theta) \cdot  v =   \sqrt{\epsilon}\, \mathfrak{M}^{-1}\bigg\langle J(\theta)\cdot  v,\Phi'(\theta)\bigg\rangle_{\theta}
\\
&\quad \sqrt{\epsilon}\, \mathfrak{M}^{-2}\bigg\langle F\big(\Phi(\theta)),\Phi'( \theta) \bigg\rangle_{\theta} \left [\bigg\langle  v, \Phi''(\theta) \bigg\rangle_{\theta} +\left \langle v,\left . \frac{d}{da}\big[P^{-\top}(a)P^{-1}(a)\big]\right |_{a=\theta}\Phi'(\theta)\right \rangle  \right ] \\
&=\sqrt{\epsilon}\, \bigg\langle J(\theta)\cdot  v,\Phi'(\theta) \bigg\rangle_{\theta}+\sqrt{\epsilon}\,  \omega_0\frac{d}{d\theta} \bigg\langle  v, \Phi'(\theta) \bigg\rangle_{\theta}, 
\end{align*}
since $\M(\theta)=1$ and $\bigg\langle F\big(\Phi(\theta)),\Phi'( \theta) \bigg\rangle_{\theta}=\omega_0$.
Observe that
\begin{align*}
\bigg\langle J(\theta)\cdot  v,\Phi'(\theta) \bigg\rangle_{\theta}&=\bigg\langle P^{-1}(\theta)J(\theta)\cdot  v,P^{-1}(\theta)\Phi'(\theta) \bigg\rangle \\
&=\bigg\langle J(\theta)\cdot  v,[P(\theta)P^{\top}(\theta)]^{-1}\Phi'(\theta) \bigg\rangle\\
&=\bigg\langle   v,J(\theta)^{\top}\cdot R(\theta) \bigg\rangle\\
&=-\omega_0 \bigg\langle   v, R'(\theta) \bigg\rangle\\
&=-\omega_0 \bigg\langle v, \frac{d}{d\theta}\bigg\lbrace\big[P^{-\top}(\theta)P^{-1}(\theta)\big]\Phi'(\theta)\bigg\rbrace\bigg\rangle \\
&=- \omega_0\frac{d}{d\theta} \bigg\langle  v, \Phi'(\theta) \bigg\rangle_{\theta},
\end{align*}
where in the third last line we have used \eqref{adj}, and in the second last line we have used \eqref{eq: PRC again}.

We have thus proven that the phase equation decouples from the amplitude equation at $O(\sqrt{\epsilon})$, which is consistent with the analysis of \cite{Bonnin17}. Since the errors in the SDE are of $O(\epsilon)$, this linearization of our phase equation is accurate over timescales of order $O(\epsilon^{-1})$, which is the same order as the isochronal phase equation.

\section{Explicit bounds on the growth of the weighted amplitude $\norm{u_t-\Phi(\beta_t)}_{\beta_t}$}

In this section we obtain powerful bounds on how long it takes the weighted amplitude of the orthogonal fluctuations, $\norm{u_t-\Phi(\beta_t)}_{\beta_t}$ to exceed some value $a$. These bounds are valid for $\norm{u_t-\Phi(\beta_t)}_{\beta_t} = o(b)$, where $b$ is the magnitude of the decay of transverse fluctuations, and are useful in a variety of situations. One situation is in the limit of small noise as $\epsilon \to 0$.  Another situation where these bounds are useful is in the regime of finite noise (so we do not take $\epsilon \to 0$), but a large decay of fluctuations that are transverse to the limit cycle (i.e. large $b$) \cite{Teramae09,Newby14}. These bounds are more powerful and flexible than classical large deviations bounds, because both the neighborhood $[0,a]$ and the time interval $T$ can vary with $\epsilon$ and $b$. The relative simplicity of the proof of this theorem provides further justification for the phase decomposition outlined in the first half of this paper. It results in a uniform lower bound for the decay of the transformed drift $w_t = P\big(\beta_t\big)^{-1}v_t$, which means that after a rescaling of time using the Dambins-Dubins-Schwarz theorem \cite{revuz2013continuous}, it becomes straightforward to demonstrate that the amplitude term behaves like a stable Ornstein-Uhlenbeck Process. This theorem can also be used to bound the probability of the stopping time $\tau$ (defined in \eqref{defn: tau}) exceeding a certain value.

 The following bounds are expressed in terms of the first hitting time of the scalar Ornstein-Uhlenbeck Process, which we restate here. Let $p^{(-b)}_{x,a}(t)$ be the density for the first hitting time of the Ornstein Uhlenbeck process with drift gradient $-b$ started at $x$. More precisely, if 
\begin{align}
dY_t =& -b Y_t dt + dW_t,\quad 
Y_0 = x,
\end{align}
for a one-dimensional Brownian Motion $W$, then 
\begin{equation}
\mathbb{P}\bigg( \inf\lbrace s>0 : Y_s = \kappa\rbrace \in [t,t+dt] \bigg) := p^{(-b)}_{x,\kappa}(t)dt.
\end{equation}
Let $I(\epsilon,b) \subset \mathbb{R}^+$ be the following closed interval
\begin{multline}\label{eq: a range}
I(\epsilon,b) = \bigg\lbrace a \in \mathbb{R}^+:C_1 \epsilon + C_2 a^2 \leq \frac{ba}{2} \text{ and }\\
a \leq \frac{1}{2}\bigg(\sup_{\alpha \in [0,2\pi]}\norm{\Phi''(\alpha) - \frac{d}{d\theta}\big[P(\theta)P^{\top}(\theta)\big]\big|_{\theta=\alpha}[P(\alpha)P^{\top}(\alpha)]^{-1}\Phi'\big(\alpha\big)}_{\alpha}\bigg)^{-1}\bigg\rbrace,
\end{multline}
where $C_1$ and $C_2$ are positive constants (independent of $\epsilon$ and $b$) that are specified in Lemma \ref{Lemma Bound gamma2}. The second condition in the above definition is to ensure that the SDE for $\beta_t$ is well-defined as long as $\norm{u_t - \Phi(\beta_t)}_{\beta_t} \in I(\epsilon,b)$.

The following theorem obtains bounds on how long it takes $\norm{u_t-\Phi(\beta_t)}_{\beta_t}$ to attain any $a$ in $I(\epsilon,b)$. The theorem is most useful in the regime $a\in \big( O\big(\sqrt{\frac{\epsilon}{b}}\big), O(b)\big)$. It is not very useful for values of $a$ towards the lower end of $I(\epsilon,b)$, since $\norm{u_t-\Phi(\beta_t)}_{\beta_t}$ will very quickly attain $O\big(\sqrt{\frac{\epsilon}{b}}\big)$, since in this regime the fluctuations of the noise dominate the $-b$ decay resulting from the stability of the deterministic dynamics. 

Recall that $\tau$ (defined in \eqref{defn: tau}) is the stopping time such that the SDE for the phase in \S 2 is well-defined for all $t \leq \tau$.
\begin{theorem}\label{Theorem: Bound v small}
For all $a \in I(\epsilon,b)$, if 
\begin{equation}
\sup_{s\in [0,T]}\norm{u_s -\Phi(\beta_s)}_{\beta_s} \leq a,
\end{equation}
then $T\leq \tau$. Furthermore, if $ \norm{u_0- \Phi(\beta_0)}_{\beta_0} := x < \frac{a}{2}$, then 
\begin{equation}
\mathbb{P}\bigg(\sup_{s\in [0,T]}\norm{u_s -\Phi(\beta_s)}_{\beta_s} \geq a \bigg) \leq \int_0^T p^{(-b)}_{\bar{x},\bar{a}}(u)du,
\end{equation}
where $\bar{x} = {x}/{ \sqrt{\lambda \epsilon}}$ and $\bar{a} ={a}/{2\sqrt{\lambda\epsilon}}$, and $\lambda $ is a positive constant that is given in \eqref{defn lambda}. Note that $\lambda$ is determined by $\Pi$, $G$ and $Q$.
\end{theorem}
\begin{rem}
To facilitate the exposition, we have chosen $\bar{a} = {a}/{2\sqrt{\lambda\epsilon}}$. In fact, we could have chosen $\bar{a} = {\rho a}/{\sqrt{\lambda \epsilon}}$,  for any $\rho \in (0,1)$, and the bound would still hold in the limit $\epsilon / b \to 0$.
\end{rem}

\begin{rem}
We can use classical results on the first hitting time of the Ornstein-Uhlenbeck process to derive the leading order asymptotics of the above \cite{Ricciardi88,Alili05}. To leading order, as  $ {b}/{\epsilon} \to \infty$ , 
\begin{equation}
p^{(-b)}_{0,\bar{a}}(t) \simeq bg\bigg(\frac{a^2b}{4\lambda\epsilon} \bigg)\exp\bigg\lbrace - btg\bigg(\frac{a^2b}{4\lambda\epsilon} \bigg) \bigg\rbrace,
\end{equation}
where $g(z) = \frac{\sqrt{z}}{\sqrt{2\pi}}\exp\big\lbrace - \frac{z}{2}\big\rbrace$. We find that for $a\in I(\epsilon,b)$,
%\[
%\mathbb{P}\bigg(\sup_{s\in [0,\mathfrak{T}]}\norm{u_s -\Phi(\beta_s)}_{\beta_s} \geq a \bigg) \simeq \exp\bigg\lbrace - \frac{b}{2}g\bigg(\frac{a^2b}{4\lambda\epsilon} \bigg) \bigg\rbrace,
%\]
%\[
%\mathbb{P}\bigg(\sup_{s\in [0,T]}\norm{u_s -\Phi(\beta_s)}_{\beta_s} \geq a \bigg) \simeq T\exp\bigg\lbrace - bg\bigg(\frac{a^2b}{4\lambda\epsilon} \bigg) \bigg\rbrace,
%\]
if $T = o\big(g\big\lbrace\frac{a^2b}{4\lambda\epsilon} \big\rbrace^{-1} \big)$, then $\mathbb{P}\big(\sup_{s\in [0,T]}\norm{u_s -\Phi(\beta_s)}_{\beta_s} \geq a \big)\simeq Tg\big\lbrace\frac{a^2b}{4\lambda\epsilon} \big\rbrace  \ll 1$. There are much more refined estimates in the literature: note in particular the exact analytic expression in \cite[Theorem 3.1]{Alili05}. 
\end{rem}

\section{Discussion and Future Work}

In summary, the variational approach developed in this paper determines the phase of a stochastic oscillator by requiring it to minimize a weighted norm. We have demonstrated that to leading order, the phase separates from the amplitude and agrees with the isochronal phase. Hence, the linearization of our phase dynamics is accurate over timescales of $O(\epsilon^{-1})$, which is the same order of accuracy as the isochronal phase equation. In addition, our exact phase equation \eqref{bee} is accurate over much longer timescales of order $O\big(\exp(Cb\epsilon^{-1})\big)$, recalling that $b$ is the rate of decay of transverse fluctuations. There exists a precise analytic expression for the phase SDE, as well as a stopping time $\tau$ up to which this SDE applies. Furthermore, one can immediately determine the phase from any particular realization of the fundamental SDE using \eqref{eq: minimum G 2} (as long as one takes the phase to be the global minimum). This is an advantage of our method compared to the isochronal method, since in most cases there does not exist an analytic solution for the isochronal method, and it is difficult to implement in a computationally efficient way \cite{Ashwin16}. 

The phase SDE \eqref{bee} is thus very well-suited to studying the long-time dynamics of the phase on timescales of $O\big(\exp(Cb\epsilon^{-1})\big)$. In \S 4 we obtained powerful bounds on the probability of the oscillator leaving any particular neighborhood of the oscillator over any particular timescale. These bounds are very flexible, because they shed light on the mutual scaling of the amplitude of the noise, rate of decay of transverse fluctuations, the size of the neighborhood of the limit cycle and the time the oscillator spends in this neighborhood.

In forthcoming work, we will use the phase SDE of this paper to study the synchronization of uncoupled oscillators subject to common noise. In particular, we will obtain precise bounds on the probability of two synchronized oscillators desynchronizing, and conditions under which two oscillators never desynchronize. Another interesting application of the phase SDE of this paper would be the effect of finite noise on oscillators with a strong decay of transverse fluctuations \cite{Newby14}.

\newpage

\begin{appendix}
\section{Proof of Theorem \ref{Theorem: Bound v small}}

\begin{proof}
We start with the first part of the theorem. From \eqref{curve},
\begin{align*}
\M(z,\theta)=& 1 - \bigg\langle z-\Phi(\theta), \Phi''(\theta)\bigg\rangle_\theta - \left\langle z-\Phi(\theta),P(\theta)P^{\top}(\theta)\frac{d}{d\theta}\big[{P(\theta)P^{\top}(\theta)}\big]^{-1}\Phi'(\theta)\right \rangle_{\theta} \\
=& 1 - \left\langle z-\Phi(\theta), \Phi''(\theta) - \frac{d}{d\theta}\big[{P(\theta)P^{\top}(\theta)}\big]\big[{P(\theta)P^{\top}(\theta)}\big]^{{-1}}\Phi'(\theta)\right \rangle_{\theta},
\end{align*}
and through an application of the Cauchy-Schwarz Inequality to the above, it may be observed that
\[
\M(u_t,\beta_t) \geq 1 - \norm{u_t-\Phi(\beta_t)}_{\beta_t}\norm{\Phi''(\beta_t) - \frac{d}{d\theta}\big[P(\theta)P^{\top}(\theta)\big]\big|_{\theta=\beta_t}P^{-\top}(\beta_t)P^{-1}(\beta_t)\Phi'\big(\beta_t\big)}_{\beta_t}.
\]
It then follows from the definition of $I(\epsilon,b)$ that if $\sup_{s\in [0,t]}\norm{u_s -\Phi(\beta_s)}_{\beta_s} \leq a$, for $a\in I(\epsilon,b)$, then $\M(u_s,\beta_s) \geq \frac{1}{2}$ for all $s\in [0,t]$, and therefore
\begin{equation}\label{eq: M bounded}
\sup_{s\in [0,t]}\M(u_s,\beta_s)^{-1} \leq 2.
\end{equation}
This means that 
\begin{equation}\label{eq: inequality tau}
\tau \geq \inf\big\lbrace s\geq 0: \norm{u_s - \Phi(\beta_s)}_{\beta_s} = a\big\rbrace.
\end{equation}
In other words, the SDE for the phase $\beta_t$ that we derived in \S 2 is well-defined as long as $\norm{u_t -\Phi(\beta_t)}_{\beta_t} \leq a$.

We now prove the second part of the theorem.
Recall that the amplitude term satisfies equation (\ref{vee}). In the following it is convenient to perform the rescaling $\sqrt{\epsilon} v_t\rightarrow v_t$. \begin{align}\label{eq: first vt}
dv_t=\left [J\big(\beta_t\big)v_t + \gamma_0(u_t,\beta_t)\right ]dt + \sqrt{\epsilon} \tilde{G}(u_t,\beta_t)dW_t 
\end{align}
where
\begin{multline}\label{eq:gamma0}
\gamma_0(u_t,\beta_t) = F(u_t)-J(\beta_t)v_t- \mathfrak{M}(u_t,\beta_t)^{-1}\Phi'(\beta_t)\left (\big\langle  F(u_t),\Phi'(\beta_t) \big\rangle_{\beta_t}+\epsilon \kappa(u_t,\beta_t)\right ) \\
 -\frac{\epsilon}{2} \Phi''(\beta_t)\M(u_t,\beta_t)^{-2} \left \langle K(u_t,\beta_t)\Phi'(\beta_t),QK(u_t,\beta_t) \Phi'(\beta_t) \right \rangle
\end{multline}
and $\tilde{G}(u_t,\beta_t) \in \mathbb{R}^{d\times d}$ is given by
\[
\tilde{G}(u_t,\beta_t) = G(u_t) -   \M(u_t,\beta_t)^{-1} \Phi'(\beta_t)\Phi'(\beta_t)^{\top}P^{-\top}\big(\beta_t\big)P^{-1}\big(\beta_t\big)G(u_t).
\]
We now perform the change of variable $w_t = P\big(\beta_t\big)^{-1}v_t$, since $\norm{v_t}_{\beta_t} = \norm{w_t}$. Using Ito's Lemma, we find that
\begin{multline}
dw_t = -P\big(\beta_t\big)^{-1}P'\big(\beta_t \big)w_t d\beta_t + P\big(\beta_t\big)^{-1}dv_t - P\big(\beta_t\big)^{-1}P'\big(\beta_t \big)P\big(\beta_t\big)^{-1}dv_t d\beta_t .
\end{multline}
As will be seen further below, the reason for this change of variable is that the drift of $w_t$ decays uniformly (to leading order), so that the leading order behavior of the SDE is like a stable Ornstein-Uhlenbeck process. We now demonstrate this. Recall from \eqref{eq: derivative P 0} that the derivative of $P(t)$ satisfies
\begin{equation}\label{eq: derivative P}
\omega_0 P'(t) = J(t)P(t)- P(t)\mathcal{S}
\end{equation}
This means that
\begin{multline*}
dw_t = \omega_0^{-1}\big(-P\big(\beta_t\big)^{-1}J(\beta_t)v_t +\mathcal{S}w_t\big)d\beta_t + P\big(\beta_t\big)^{-1}dv_t - P\big(\beta_t\big)^{-1}P'\big(\beta_t \big)P\big(\beta_t\big)^{-1}dv_t d\beta_t,
\end{multline*}
and therefore
\begin{multline}\label{wee0}
dw_t = \big[\mathcal{S}w_t + \gamma(u_t,\beta_t) \big]dt+\sqrt{\epsilon} P^{-1}(\beta_t)\tilde{G}(u_t,\beta_t)dW_t \\ + \sqrt{\epsilon}\omega_0^{-1}\mathfrak{M}^{-1}(u_t,\beta_t) \big\lbrace-P\big(\beta_t\big)^{-1}J(\beta_t)v_t +\mathcal{S}w_t\big\rbrace\Phi'(\beta_t)^{\top}P^{-\top}(\beta_t)P^{-1}(\beta_t)G(u_t) dW_t, 
\end{multline}
where
\begin{multline}
\gamma(u_t,\beta_t) = P(\beta_t)^{-1}\big(J(\beta_t)v_t + \gamma_0(u_t,\beta_t)\big) -\mathcal{S}w_t \\
+\omega_0^{-1} \big(-P\big(\beta_t\big)^{-1}J(\beta_t)v_t +\mathcal{S}w_t\big)\mathfrak{M}(u_t,\beta_t)^{-1}\left( \big\langle  F(u_t),\Phi'(\beta_t) \big\rangle_{\beta_t}+\epsilon \kappa(u_t,\beta_t)\right ) \\
- \epsilon\mathfrak{M}^{-1}(u_t,\beta_t) P\big(\beta_t\big)^{-1}P'\big(\beta_t \big)P\big(\beta_t\big)^{-1}\tilde{G}\big(u_t,\beta_t\big)Q G^{\top}(u_t)P^{-\top}(\beta_t)P^{-1}(\beta_t)\Phi'(\beta_t).
\end{multline}
and we have used the fact that 
\begin{align*}
d\beta_t =& \sqrt{\epsilon}\mathfrak{M}^{-1}(u_t,\beta_t)\big\langle P^{-1}(\beta_t)\Phi'(\beta_t),P^{-1}(\beta_t)G(u_t)dW_t\big\rangle + F.V.T\\
=& \sqrt{\epsilon}\mathfrak{M}^{-1}(u_t,\beta_t) \Phi'(\beta_t)^{\top} P^{-\top}(\beta_t)P^{-1}(\beta_t)G(u_t)dW_t + F.V.T,
\end{align*}
where $F.V.T$ stands for `finite variation terms' (i.e. the drift terms). We write this as
\begin{equation}\label{wee}
dw_t = \bigg[\omega_0\mathcal{S}w_t + \gamma(u_t,\beta_t) \bigg]dt+\sqrt{\epsilon} \bar{G}(u_t,\beta_t)dW_t,
\end{equation}
where $\bar{G}(u_t,\beta_t)$ can be inferred from \eqref{wee0}.

Since the map $w \to \norm{w}^2$ is twice differentiable, we can apply Ito's Lemma to equation (\ref{wee}). We find that
\begin{multline}
d\norm{w_t}^2 =\bigg[2 \big\langle w_t,\mathcal{S}w_t +\gamma(u_t,\beta_t,t)\big\rangle +\epsilon \rm{tr}\big\lbrace  \bar{G}(u_t,\beta_t)Q\bar{G}^{\top}(u_t,\beta_t) \big\rbrace \bigg]dt \\+2\sqrt{\epsilon}\bigg\langle w_t,\bar{G}(u_t,\beta_t)dW_t \bigg\rangle .
 \end{multline}
 It follows from the stability assumption at the start of this paper that $\big\langle w_t,\mathcal{S}w_t\big\rangle \leq - b\norm{w_t}^2$, which means that
\begin{equation}
d\norm{w_t}^2 \leq  \big[-2b\norm{w_t}^2 +2\gamma_2(u_t,\beta_t) \big]dt + 2\sqrt{\epsilon} \langle w_t, \bar{G}(u_t)dW_t\rangle ,\label{eq SDE vt beta_t}
 \end{equation}
 where 
 \[
\gamma_2(u_t,\beta_t,t) = \langle w_t , \gamma(u_t,\beta_t)\rangle +\frac{\epsilon}{2} \rm{tr}\big\lbrace  \bar{G}(u_t,\beta_t)Q\bar{G}^{\top}(u_t,\beta_t) \big\rbrace .
 \]

Define the stopping time 
\begin{equation} 
\hat{\tau}_{a} = \inf\bigg\lbrace s \leq \tau :\norm{w_s}^{-1}\gamma_2\big(u_s,\beta_s\big) = ba/2\bigg\rbrace,
\end{equation}
recalling that $\tau$ (defined in \eqref{defn: tau}) is the stopping time for which the SDE for $\beta_t$ is well-defined).

We determine an SDE for $\norm{w_t}$ by applying Ito's Lemma to the square root function, finding that for all times $t \leq \hat{\tau}_a$
 \begin{multline}
d\norm{w_t} \leq \sqrt{\epsilon}\norm{w_t}^{-1} \big\langle w_t,\bar{G}(u_t)dW_t \big\rangle \\ + \bigg(-b \norm{w_t}+ \norm{w_t}^{-1}\gamma_2(u_t,\beta_t,t)- \frac{\epsilon}{4\norm{w_t}^{3}}\big\langle Q\bar{G}^{\top}(u_t)w_t,\bar{G}^{\top}(u_t) w_t  \big\rangle \bigg) dt\\
\leq \sqrt{\epsilon}\norm{w_t}^{-1} \big\langle w_t,\bar{G}(u_t)dW_t \big\rangle + \bigg(-b \norm{w_t}+ \norm{w_t}^{-1}\gamma_2(u_t,\beta_t,t) \bigg) dt,\label{eq SDE vt beta_t 5}
 \end{multline}
since  $\frac{\epsilon}{4\norm{w_t}^{3}}\big\langle Q\bar{G}^{\top}(u_t) w_t,\bar{G}^{\top}(u_t) w_t  \big\rangle \geq 0$, because the covariance matrix $Q$ is positive semi-definite. Note that the coefficients of the above SDE are continuous and bounded in a sufficiently small neighborhood of $\norm{w_t} = 0$. This is true for $\norm{w_t}^{-1}\gamma_2$ thanks to the inequality in Lemma \ref{Lemma Bound gamma2}, and it is true for the diffusion term thanks to the Cauchy-Schwarz Inequality (this will be clear in the following).

 Now define $y_t = \exp\big(b t\big)\norm{w_t}$. Through Ito's Lemma, we find that
\[
dy_t = by_t dt +\exp\big(bt\big)d\norm{w_t},
\]
and therefore for all times $t \leq \hat{\tau}_a$,
\begin{multline*}
dy_t  \leq  \exp\big(bt\big)\bigg\lbrace b \norm{w_t}-b \norm{w_t}+ \norm{w_t}^{-1}\gamma_2(u_t,\beta_t,t) \bigg\rbrace dt\\ +\sqrt{\epsilon} \norm{w_t}^{-1} \exp\big(bt\big)\big\langle w_t,\bar{G}(u_t,\beta_t)dW_t\big\rangle.
\end{multline*}
We integrate the above expression, before dividing both sides by $\exp(bt)$, and find that
\begin{align*}
\norm{w_{t\wedge \hat{\tau}_a}} \leq  \exp\big\lbrace-b(t\wedge \hat{\tau}_a)\big\rbrace\norm{w_0} +\sqrt{\epsilon} \int_0^{t\wedge \hat{\tau}_a}\exp\big(b(s-t\wedge \hat{\tau}_a) \big)\norm{w_s}^{-1}\big\langle w_s,\bar{G}(u_s)dW_s\big\rangle\\+ \int_0^{t\wedge \hat{\tau}_a} \exp\big(b(s-t\wedge \hat{\tau}_a) \big)\norm{w_s}^{-1}\gamma_2(u_s,\beta_s) ds.
\end{align*}
This means that
\begin{multline}
\norm{w_{t\wedge \hat{\tau}_a}} \leq  \exp\big\lbrace-b(t\wedge \hat{\tau}_a) \big\rbrace\norm{w_0} + \frac{1}{b}\sup_{s\in [0,t\wedge \hat{\tau}_a]}\norm{w_s}^{-1}\big|\gamma_2(u_s,\beta_s)|\\ +\sqrt{\epsilon} \int_0^{t\wedge \hat{\tau}_a}\norm{w_s}^{-1}\exp\big(b(s-t\wedge \hat{\tau}_a) \big)\big\langle w_s,\bar{G}(u_s,\beta_s)dW_s\big\rangle\\
 \leq  \exp\big\lbrace-b(t\wedge \hat{\tau}_a) \big\rbrace\norm{w_0} + \frac{a}{2} +\sqrt{\epsilon} \int_0^{t\wedge \hat{\tau}_a}\norm{w_s}^{-1}\exp\big\lbrace b(s-t\wedge \hat{\tau}_a) \big\rbrace\big\langle w_s,\bar{G}(u_s,\beta_s)dW_s\big\rangle,
\label{eq: first bound norm vt}
\end{multline}
using the definition of $\hat{\tau}_a$.
%Substituting the definition of $Y_t$ in ,
%\begin{multline}
%\norm{v_t} \leq  \exp\big(-bt \big)\norm{v_0} + \frac{1}{b}\sup_{s\in [0,t]}\norm{v_s}^{-1}\bigg|\bigg\langle v_s,K(v_s,\beta_s)\bigg\rangle + \frac{\epsilon}{2}  \text{Tr}\big(G(u_s)QG^{\top}(u_s)\big)\bigg|\\ +\int_0^t \norm{v_s}^{-1}\exp\big(b(s-t) \big)\big\langle v_s,G(u_s)\Gamma Y_s\big\rangle ds+\sqrt{\epsilon} \int_0^t \norm{v_s}^{-1}\exp\big(b(s-t) \big)\big\langle v_s,G(u_s)dW_s\big\rangle ds.\\
%\label{eq: first bound norm vt}
%\end{multline}

Define the stopping time
\begin{equation}
\tau_{a,x} = \inf\big\lbrace \hat{\tau}_{a},\grave{\tau}_{a,x}\big\rbrace%\overline{\tau}_{a,x},
\end{equation}
where
\begin{align}%\overline{\tau}_{a,x} =& \inf\bigg\lbrace s \geq 0 :\int_0^s\norm{v_t}^{-1}\exp\big(b(t-s) \big)\big\langle v_t,G(u_t)\Gamma Y_t \big\rangle dt = a/3\bigg\rbrace\\
\grave{\tau}_{a,x} =&  \inf\bigg\lbrace s \geq 0 :x\exp\big(-bs\big)+\sqrt{\epsilon}\int_0^s\norm{w_t}^{-1}\exp\big(b(t-s) \big)\big\langle w_t,\bar{G}(u_t)dW_t\big\rangle = a/2\bigg\rbrace
\end{align} 
It follows from \eqref{eq: first bound norm vt} that for all $s\in [0,\tau_{a,x}]$,
\begin{equation}
\norm{w_s} \leq a.\label{eq: bound norm vs}
\end{equation}
%The following theorem is useful because it provides a powerful bound on the probability of the system leaving the neighborhood of a limit cycle. In particular, in the $\epsilon \ll 1$ regime, we find that the probability scales as $\exp\big(-C\epsilon^{-1}\big)$. 
%Let
%\begin{equation}\label{definition: C_G}
%C_G = \sup_{X \in \R^d}\rm{Tr}\bigg(G(X)G^{\top}(X)G(X)G^\top(X)\bigg)^{\frac{1}{2}}.
%\end{equation}
%Let $\mathfrak{a}(\epsilon)$ be the greatest $a > 0$ such that 
%\begin{equation}\label{eq:SDE a epsilon}
%\sqrt{\epsilon} C_G \text{Tr}\big(Q^2\big)^{\frac{1}{2}} + Ca^{2} = \frac{ba}{3}, 
%\end{equation}
%where $C_G$ is defined in \eqref{definition: C_G} and $C$ is defined in \eqref{definition C}. 
%In deriving the SDE \eqref{eq: first bound norm vt} bounding $\norm{v_t}$, we assumed that $\norm{v_t} \geq \epsilon$. This assumption does not affect the theorem because we are only interested in the first hitting time of $a > \epsilon$. More precisely, if at any time $\norm{v_t} < \epsilon$, then we could always stop the SDE bounding $\norm{v_t}$ and restart it once $\norm{v_t}$ reattains $\epsilon$. To simplify these proofs we avoid going into this formalism and simply assume that $\norm{v_t} \geq \epsilon$ always.
%Noting the identity in \eqref{eq: bound norm vs}, $\norm{v_s} \leq a$ for all $s\leq \tau_{a,x}$. 
This means that
\begin{multline}\label{eq: list probabilities}
\mathbb{P}\big( \tau_{a,x} \leq T\big) \leq
\mathbb{P}\bigg(\text{There exists }s\in [0,T]\; \text{ such that either }\;\zeta_s- x \geq \exp(bs)\frac{a}{2}\\ \text{ or } \bigg|\frac{1}{\norm{w_s}}\gamma_2(u_s,\beta_s)\bigg| =ba/2\text{, and }\sup_{r\in [0,s]}\norm{w_r} \leq a \bigg) \\
\leq \mathbb{P}\bigg(\text{There exists }s\in [0,T]\; \text{ such that }\;\zeta_s- x \geq \exp(bs)a/2\bigg)\\ +
 \mathbb{P}\bigg(\text{There exists }s\in [0,T] \text{ such that }\big| \gamma_2(u_s,\beta_s) \big| =ba/2\\ \text{ or }\tau \leq T\tau\text{, and }\sup_{r\in [0,s]}\norm{w_r} \leq a\bigg),
 \end{multline}
where $\zeta_s = \sqrt{\epsilon}\int_0^s\norm{w_t}^{-1}\exp\big(bt \big)\big\langle w_t,\bar{G}(u_t)dW_t\big\rangle$. 

Now it follows from \eqref{eq: inequality tau} that 
\[
\mathbb{P}\bigg(\tau \leq T\text{ and }\sup_{s\in [0,T]}\norm{w_s} \leq a\bigg) = 0.
\]
Furthermore, it follows from Lemma \ref{Lemma Bound gamma2} that 
\begin{align*}
 \mathbb{P}\bigg(\text{There exists }s\in [0,\tau] \text{ such that }\big| \norm{w_s}^{-1}\gamma_2(u_s,\beta_s) \big)\big| =ba/2 \text{ and }\sup_{t\in [0,s]}\norm{w_t} \leq a\bigg) \\
\leq \mathbb{P}\bigg(\text{There exists }s\in [0,\tau] \text{ such that }C_1 \epsilon + C_2 \norm{w_s}^2 =ba/2 \text{ and }\sup_{t\in [0,s]}\norm{w_t} \leq a\bigg) \\= 0,
 \end{align*} 
 thanks to the fact that $a\in I(\epsilon,b)$, which we recall is defined in \eqref{eq: a range}.
It therefore remains for us to prove that
\begin{multline}\label{eq: to prove, final theorem}
\mathbb{P}\bigg(\text{There exists }s\in [0,T]\; \text{such that}\;\zeta_s- x \geq \exp(bs)\frac{a}{2} \bigg) \leq \int_0^T p^{(-b)}_{\bar{x},\bar{a}}(y)dy,
\end{multline}
recalling that $\bar{a} =  {a}/{2\sqrt{\lambda \epsilon}}$ and $\bar{x} ={x}/{\sqrt{\lambda \epsilon}}$.

By the Dambis Dubins-Schwarz Theorem \cite[Theorem 1.6, Page 181]{revuz2013continuous}, $X_t := \zeta_{\iota_t}$ is Brownian, where 
\begin{align}
\iota_s =& \inf\big\lbrace r\geq 0:\eta_r \geq s \big\rbrace \text{ and } \\
\eta_r :=& \epsilon\int_0^r \norm{w_s}^{-2}  \exp\big(bs\big)\big\langle \bar{G}^{\top}(u_s)w_s,Q\bar{G}^{\top}(u_s)w_s\big\rangle ds.
\end{align}
%Let 
%\begin{equation}\label{eq: lambda eigenvalue definition}
%\lambda = \sup_{s\in [0,T]}\sup_{\mathfrak{v}_s:\norm{\mathfrak{v}}=1} \big\langle \mathfrak{v},G^{\top}(u_s)QG^{\top}(u_s)\mathfrak{v}\big\rangle
%\end{equation}
Let $\lambda_{G,P}$ be an upper bound for $\norm{\bar{G}(u_t,\beta_t)}$ (the spectral norm), that is uniform over all $\beta_t \in \mathbb{R}$ and $u_t \in \mathbb{R}^d$, recalling the implicit definition of $\bar{G}$ in \eqref{wee}. Such an upper bound exists, because by assumption $\norm{G(u_t)}$ possesses a uniform upper bound. Similarly $P^{-1}(\beta_t)$ and $J(\beta_t)$ possess uniform upper bounds because they are continuous and $2\pi$ periodic. Since $\bar{G}(u_t,\beta_t)$ is equal to sums and multiplications of matrices with uniform upper bounds, it must also possess a uniform upper bound. It follows that
\begin{align}
\big\langle \bar{G}^{\top}(u_s)w_s,Q\bar{G}^{\top}(u_s)w_s\big\rangle
=&\big\langle w_s,\bar{G}(u_s)Q\bar{G}^{\top}(u_s)w_s\big\rangle\nonumber \\
%=\big\langle v_s,G^{\top}(u_s)QG^{\top}(u_s) v_s\big\rangle
\leq & \lambda \norm{w_s}^2,\label{defn lambda}
\end{align}
where $\lambda = \lambda_{G,P}^2 \lambda_Q$. We find that $\eta_r \leq \bar{\eta}_r :=\frac{\epsilon}{b}\big\lbrace \exp\big(br\big)-1\big\rbrace\lambda $. Writing $\bar{\iota}_s = \inf\big\lbrace r\geq 0:\bar{\eta}_r \geq s \big\rbrace$, we have that $\bar{\iota}_s \leq \iota_s$, and
\begin{align}
&\mathbb{P}\bigg( \text{There exists }r\in [0,T]\; ,\;\zeta_r- x \geq \exp(br)\frac{a}{2}  \bigg)\nonumber\\
&= \mathbb{P}\bigg( \text{There exists }r\in [0,T]\; ,\; X_{\eta_r}- x \geq \exp(br)\frac{a}{2}  \bigg)\nonumber\\
&= \mathbb{P}\bigg( \text{There exists }y\in [0,\eta_T]\; ,\; X_{y}-x \geq \exp(b\iota_y)\frac{a}{2}  \bigg)\nonumber\\
&\leq \mathbb{P}\bigg( \text{There exists }y\in [0,\bar{\eta}_T]\; ,\; X_{y}-x \geq \exp\big(b\iota_y\big)\frac{a}{2}  \bigg)\nonumber\\
&\leq \mathbb{P}\bigg(  \text{There exists }y\in [0,\bar{\eta}_T]\; ,\; X_{y}-x \geq \exp\big(b\bar{\iota}_y\big)\frac{a}{2}\bigg).\label{eq: tmp probability 2}
%&\leq \mathbb{P}\bigg(\sup_{s\in [0,T]}\norm{v_s} \leq \mathfrak{X}_a  \text{ and there exists }x\in [0,T]\; ,\; X_{\bar{\xi}^a_x}+ \norm{v_0}^2 \geq \exp(2bx)\kappa\bigg)\\
%&\leq \mathbb{P}\bigg(\text{there exists }x\in [0,T]\; ,\; X_{\bar{\xi}^a_x}+ \norm{v_0}^2 \geq \exp(2bx)\kappa\bigg).
%&\leq \mathbb{P}\bigg(\text{there exists }x\in [0,T]\; ,\; X_{\bar{\xi}^a_x}+ \norm{v_0}^2 \geq \exp(2bx)\kappa\bigg)\\
%&= \mathbb{P}\bigg(\text{there exists }x\in \big[0,\exp\big(2bT\big)\lambda / 2b\big]\; ,\; X_{x}+ \norm{v_0}^2 \geq \frac{2bx\kappa}{\lambda\mathcal{X}_a^2}\bigg)
\end{align}
Now suppose that $Z$ satisfies the SDE
\begin{align}
dZ_t =& -b Z_t dt + \sqrt{\lambda \epsilon} dW_t \\
Z_0 =& x,
\end{align}
for a $1d$ Brownian Motion $W$. The solution of this SDE is
\begin{equation}
Z_t = \exp\big(-bt\big)x+ \sqrt{\lambda  \epsilon}  \int_0^t \exp\big(b(s-t) \big)dW_s.
\end{equation}
Now let $\alpha_t =   \sqrt{\lambda \epsilon}\int_0^t \exp\big(bs \big)dW_s$, and observe that the quadratic variation of $\alpha_t$ is $\bar{\eta}_t$. This means that
\begin{align}
\mathbb{P}\bigg( \sup_{t\in [0,T]}Z_t \geq \frac{a}{2}  \bigg) = 
\mathbb{P}\bigg( \text{There exists }t \in [0,T] \text{ such that } \alpha_t+x \geq \frac{a}{2} \exp\big(bt\big) \bigg)\nonumber \\
=\mathbb{P}\bigg( \text{There exists }t \in  \big[0,\bar{\eta}_T\big] \text{ such that } \upsilon_t+x \geq \frac{a}{2} \exp\big(b\bar{\iota}_t\big) \bigg).\label{eq: tmp probability 1}
\end{align}
by the Dambis Dubis-Schwarz Theorem, since $\upsilon_t := \alpha_{\bar{\iota}_t}$ is Brownian. %, and 
%\begin{align*}
%\tilde{\iota}_s = \inf\big\lbrace r\geq 0:\mathbb{E}\big[\alpha_r^2\big] \geq s \big\rbrace.
%\end{align*}
It can be observed that the expressions in \eqref{eq: tmp probability 1} and \eqref{eq: tmp probability 2} are equal.

This means that 
\begin{multline*}
\mathbb{P}\bigg(\text{There exists }t \in [0,T] \text{ such that }Z_t \geq \frac{a}{2} \bigg) \\= \mathbb{P}\bigg(  \text{There exists }y\in [0,\bar{\eta}^a_T]\; \text{ such that }\; X_{y}- \norm{u_0-\Phi(\beta_0)}_{\beta_0} \geq \exp(b\bar{\iota}_y)\frac{a}{2}\bigg).
\end{multline*}
Now it can be seen that $\bar{Z}_t := \frac{1}{\sqrt{\lambda \epsilon}}Z_t$ is an Ornstein-Uhlenbeck process, and therefore
\begin{multline*}
\mathbb{P}\bigg(\text{There exists }t \in [0,T] \text{ such that }Z_t \geq \frac{a}{2} \bigg)\\ =\mathbb{P}\bigg(\text{There exists }t \in [0,T] \text{ such that }\bar{Z}_t \geq \frac{a}{2\sqrt{\epsilon\lambda}} \bigg) 
= \int_0^T p^{(-b)}_{\bar{x},\bar{a}}(s)ds, 
\end{multline*}
and we have proved the required bound in \eqref{eq: to prove, final theorem}.
\end{proof}
\medskip

 \begin{lemma}\label{Lemma Bound gamma2}
 There exist positive constants $C_1$ and $C_2$ such that, as long as $\norm{w_t} \leq a \in I(\epsilon,b)$,
 \[
 \big|\gamma_2(u_t,\beta_t)\big| \leq C_1 \norm{w_t} \epsilon + C_2 \norm{w_t}^3
 \]
 \end{lemma}
 \begin{proof}
We can decompose $\gamma_2 = \gamma^1_2 + \epsilon \gamma^2_2$: where $\gamma^1_2$ comprises higher order-corrections to the linearized behaviour, and $\gamma^2_2$ arises from quadratic and cross-variations. In the following equations, since $v_t = P(\beta_t)^{-1}w_t$, and $P(\beta_t)^{-1}$ is continuous on $\mathbb{S}^1$, it must be the case that for come constant $C_P$, $\sup_{\theta\in [0,2\pi]}\norm{P(\theta)^{-1}} \leq C_P$, and therefore $\norm{v_t} \leq C_P \norm{w_t}$. Using the definitions in \eqref{eq:gamma0}, the higher order corrections to the linearized behavior are
 \begin{multline*}
\gamma^1_2 = \bigg\langle w_t , P(\beta_t)^{-1}\bigg\lbrace F(u_t) - \mathfrak{M}(u_t,\beta_t)^{-1}\Phi'(\beta_t)\big\langle  F(u_t),\Phi'(\beta_t) \big\rangle_{\beta_t} \\
- J\big(\beta_t\big)v_t \mathfrak{M}(u_t,\beta_t)^{-1}\big\langle  F(u_t),\Phi'(\beta_t) \big\rangle_{\beta_t}\bigg\rbrace\bigg\rangle \\
+\omega_0^{-1} \big\langle w_t, \mathcal{S}w_t\big\rangle\bigg( \omega_0 -  \mathfrak{M}(u_t,\beta_t)^{-1} \big\langle  F(u_t),\Phi'(\beta_t) \big\rangle_{\beta_t}\bigg),
 \end{multline*}
 and the quadratic / cross-variation terms are
  \begin{multline}
  \gamma^2_2 =- \kappa(u_t,\beta_t) \mathfrak{M}(u_t,\beta_t)^{-1} \bigg\langle w_t , P(\beta_t)^{-1}\Phi'(\beta_t)\bigg\rangle \\-\frac{\M(u_t,\beta_t)^{-2} }{2}\bigg\langle K(u_t,\beta_t)\Phi'(\beta_t),QK(u_t,\beta_t) \Phi'(\beta_t) \bigg\rangle\bigg\langle w_t,P(\beta_t)^{-1} \Phi''(\beta_t)\bigg\rangle \\
+  \mathfrak{M}^{-1}(u_t,\beta_t) \bigg\langle w_t,P\big(\beta_t\big)^{-1}P'\big(\beta_t \big)P\big(\beta_t\big)^{-1}\tilde{G}\big(u_t,\beta_t\big)Q G^{\top}(u_t)P^{-\top}(\beta_t)P^{-1}(\beta_t)\Phi'(\beta_t)\bigg\rangle\\
+  \frac{1}{2} \rm{tr}\big\lbrace  \bar{G}(u_t,\beta_t)Q\bar{G}^{\top}(u_t,\beta_t) \big\rbrace.
  \end{multline}
We start by bounding the quadratic and cross-variation terms, i.e. $\gamma^2_2$. Now since, by assumption, $\norm{w_t} \leq a \in I(\epsilon,b)$, it follows from \eqref{eq: M bounded} that
 \begin{equation}
\M(u_s,\beta_s)^{-1} \leq 2.
\end{equation}
By assumption, there are uniform bounds for the following: $\norm{G(u_t)}$, $\norm{Q}$, $\norm{P(\beta_t)^{-1}}$, $\norm{\Phi'(\beta_t)}$, $\norm{\Phi''(\beta_t)}$, $\norm{P'(\beta_t)}$ and $\norm{P''(\beta_t)}$. The uniform bounds on the latter five matrices follows from the fact that they are continuous and $2\pi$ periodic. We can therefore apply the Cauchy-Schwarz Inequality to each of the above terms in $\gamma^2_2$, finding that for some constant $C_2 > 0$, $\big|\gamma^2_2 \big| \leq C_2\norm{w_t}$.

We now turn to bounding $\gamma^1_2$. First, it follows from the uniform boundedness of $P^{-1}$ that for some constant $C_P$,
\[
\norm{w_t} = \norm{P^{-1}(\beta_t)v_t} \leq C_P\norm{v_t}.
\]
Now it follows from \eqref{eq: minimum G 2} that $\big\langle w_t, P(\beta_t)^{-1}\Phi'(\beta_t)\big\rangle=0$, and therefore
 \[
  \bigg\langle w_t , \mathfrak{M}(u_t,\beta_t)^{-1}\Phi'(\beta_t)\big\langle  F(u_t),\Phi'(\beta_t) \big\rangle_{\beta_t} \bigg\rangle = 0,
 \]
 since this is just a scalar multiple of $\big\langle w_t, P(\beta_t)^{-1}\Phi'(\beta_t)\big\rangle$. Now we saw in the equations following \eqref{defn: H v theta} that
 \[\mathfrak{M}(u_t,\beta_t)^{-1}\big\langle  F(u_t),\Phi'(\beta_t) \big\rangle_{\beta_t} = \omega_0 + O\big(\norm{v_t}^2\big). \]
 This means that
\[
 \big\langle w_t, \mathcal{S}w_t\big\rangle\bigg( \omega_0 -  \mathfrak{M}(u_t,\beta_t)^{-1} \big\langle  F(u_t),\Phi'(\beta_t) \big\rangle_{\beta_t}\bigg) \simeq O\big(\norm{w_t}^4\big)
 \]
 It remains for us to show that
 $F(u_t) - \omega_0 J\big(\beta_t\big)v_t = O(v_t^2)$. But this follows from the multivariate Taylor Remainder Theorem, since for some $\upsilon \in [0,1]$, 
\begin{equation}
F\big(\Phi(\beta_t)+v_t\big)=F\big(\Phi(\beta_t)\big)+J\big(\beta_t\big)v_t + \frac{1}{2}F''\big(\upsilon\Phi(\beta_t) + (1-\upsilon)v_t\big)\cdot v_t \cdot v_t .
\end{equation}
By assumption, the second derivative of $F$ is uniformly bounded, and we have therefore obtained the required bound.

 \end{proof}
\end{appendix}
\bibliographystyle{siam}
%\bibliography{bibliography}

\end{document}